\documentclass[12pt]{article}
\usepackage{amsmath,amsthm,amssymb}
\usepackage{graphicx}

\newtheorem{theorem}{Theorem}[section]
\newtheorem{lemma}{Lemma}[section]

\newtheorem{definition}{Definition}[section]
\newtheorem{rem}{Remark}[section]
\newtheorem{ex}{Example}[section]

\begin{document}
\thispagestyle{empty}
\begin{center}

%%%%%%%%%%%%%%%%%%%%%%%%%%%%%%%%%%%%%%%%%%%%%%%%%%%%%%%%%%%%%%%%%%%%%%
% TITLE - please fill in the title below:
%%%%%%%%%%%%%%%%%%%%%%%%%%%%%%%%%%%%%%%%%%%%%%%%%%%%%%%%%%%%%%%%%%%%%%
{\bf\Large Stability analysis of multi-term fractional-differential equations with three fractional derivatives}

\vspace{8mm}

%%%%%%%%%%%%%%%%%%%%%%%%%%%%%%%%%%%%%%%%%%%%%%%%%%%%%%%%%%%%%%%%%%%%%%
% First author's name and affiliation:
%%%%%%%%%%%%%%%%%%%%%%%%%%%%%%%%%%%%%%%%%%%%%%%%%%%%%%%%%%%%%%%%%%%%%%

{\large Oana Brandibur$^1$,\  Eva Kaslik$^{1,2}$}

\vspace{3mm}

{\em $^1$ Dept. of Math. and Comp. Science, West University of Timi\c{s}soara, Romania\\
$^2$ Institute e-Austria Timi\c{s}oara, Romania\\
E-mail: oana.brandibur@e-uvt.ro, \ eva.kaslik@e-uvt.ro}

\vspace{5mm}
\end{center}
\vspace{8mm}
%%%%%%%%%%%%%%%%%%%%%%%%%%%%%%%%%%%%%%%%%%%%%%%%%%%%%%%%%%%%%%%%%%%%%%
% ABSTRACT: Please insert your abstract below:
%%%%%%%%%%%%%%%%%%%%%%%%%%%%%%%%%%%%%%%%%%%%%%%%%%%%%%%%%%%%%%%%%%%%%%

\textbf{Abstract:} Necessary and sufficient stability and instability conditions are obtained for multi-term homogeneous linear fractional differential equations with three Caputo derivatives and constant coefficients. In both cases,  fractional-order-dependent as well as fractional-order-independent characterisations of stability and instability properties are obtained, in terms of the coefficients of the multi-term fractional differential equation. The theoretical results are exemplified for the particular cases of the Basset and Bagley-Torvik equations, as well as for a multi-term fractional differential equation of an inextensible pendulum with fractional damping terms, and for a fractional harmonic oscillator. 

\textbf{Keywords:} multi-term fractional differential equation,  Caputo derivative,  stability, instability

 \medskip

\section{Introduction}
\label{sec:1}

Fractional calculus has been gaining increased attention over the past several decades, driven by the suitability of the nonlocal fractional-order differential operators to describe memory and hereditary properties of various real world processes \cite{Cottone,Mainardi_1996}, in comparison with their classical integer-order counterparts. Due to this fact, fractional-order operators have been employed in the mathematical modeling of various phenomena deriving from engineering, control theory, viscoelasticity, rheology,  electrochemistry, biophysics, mechanics and mechatronics, signal and image processing, etc. \cite{Diethelm_book,Kilbas,Lak,Podlubny}. As an example, in the field of neuroscience, the fractional-order of a derivative has been interpreted as the index of memory \cite{du2013measuring}. 

As in most cases, the fractional-order differential equations and systems used in the mathematical modeling of practical problems are not explicitly solvable, their  qualitative theory, and markedly the stability and asymptotic properties of their solutions are of uttermost importance, as depicted in two recent comprehensive surveys \cite{Li-survey,Rivero2013stability}. 

In the case of nonlinear commensurate fractional-order systems, Lyapunov's first method \cite{cong2020,Li_Ma_2013,Sabatier2012stability,Wang2016stability} can be successfully used to construct the linearized system in a neighborhood of an equilibrium, and then, Matignon's stability theorem \cite{Matignon} and its generalization \cite{Sabatier2012stability} provide the required information about the stability of the equilibrium. Furthermore, it has been very recently shown \cite{cong2020} that non-trivial solutions of such systems cannot converge to the equilibria faster than $t^q$, where
$q$ is the fractional order of the system. 

Nonetheless, in the case of incommensurate fractional-order systems, the stability analysis of the linearized counterparts increases in complexity, and fewer results have been reported in the literature \cite{Petras,trachtler2016bibo}.  The asymptotic properties of certain classes of linear multi-order systems of fractional differential equations (e.g. systems with block triangular coefficient matrices) have been explored in \cite{diethelm2017asymptotic}. Recently, the complete stability and instability analysis of two-dimensional incommensurate linear systems of fractional differential equations has been  presented in \cite{Brandibur_2020,Brandibur_2018}, where a characterisation of the stability properties has been obtained in terms of the main diagonal elements of the system’s matrix and its determinant.
 
In view of the close relation to fractional-order linear systems, qualitative properties of multi-term fractional-order differential equations \cite{atanackovic2014cauchy} have also been explored in recent years. It has been shown \cite{li2013equivalent} that a linear multi-term fractional-order differential equation with rational-order derivatives is equivalent to a system of fractional-order differential equations with derivatives of the same order. This common order is in fact the inverse of the least common denominator of the considered rational orders, and hence, Matignon's stability theorem proves its utility once again. As a special case, the stability of the well-known Bagley-Torvik  equation \cite{torvik1984appearance} can be analyzed along the same lines. However, to the best of our knowledge, in the presence of at least one irrational-order derivative, stability properties have only been explored in the case of two-term fractional-order differential equations  \cite{cermak2015asymptotic,cermak2015stability,jiao2012stability}.  Moreover, it is also worth noting that boundary value problems for multi-term fractional differential equations have been considered in \cite{gejji2008,luchko2011}. 

The aim of this paper is to analyze multi-term fractional-order equations with three fractional-order derivatives of Caputo type (not necessarily of rational orders). In fact, we extend recent results obtained for linear two-dimensional incommensurate fractional-order systems \cite{Brandibur_2020}, based on the similarity of the structures of the corresponding characteristic equations. As a consequence, we also establish a fractional version
of the Routh–Hurwitz criterion for linear multi-term fractional-order differential equations with three Caputo derivatives, generalizing the results presented in \cite{cermak2015asymptotic,cermak2015stability}. More precisely, we determine stability and instability conditions in terms of the coefficients of the considered multi-term fractional-order equation, as well as the considered fractional orders of the Caputo derivatives. We also deduce necessary and sufficient conditions for the asymptotic stability and instability of the fractional-order differential equation, regardless of the choice of the considered fractional orders.   

The paper is structured as follows. After enumerating some preliminaries on fractional calculus and basic definitions in section 2, the next two sections are dedicated to presenting and proving the main results of this work, namely, fractional-order dependent stability and instability results in section 3, as well as fractional-order independent stability and instability results in section 4, for the considered multi-term fractional differential equation. Examples are given in section 5, substantiating the theoretical results. Conclusions are formulated in section 6. 
 
\section{Preliminaries}
\label{sec:2}

\begin{definition}
The Caputo fractional differential operator of order $q>0$ is defined as $$^c\!D^q x(t):=\begin{cases} \displaystyle\dfrac{1}{\Gamma(n-q)}\int_0^t \dfrac{x^{(n)}(\tau)}{(t-\tau)^{q+1-n}}d\tau,& \text{if }n-1<q<n\\ \dfrac{d^nx(t)}{dt^n},& \text{if }q=n \end{cases},$$ where $n\in\mathbb{N}$ and  the Gamma function is $\Gamma(z)=\displaystyle\int_0^{\infty}y^{z-1}e^{-y}dy$, for $\Re(z)>0$. 
\end{definition}

We consider the following multi-term fractional differential equation
\begin{equation}\label{eq.multi.order}
    ^c\!D^q x(t)+\alpha ^c\!D^{q_1}x(t)+\beta ^c\!D^{q_2}x(t)+\gamma x(t)=0,
\end{equation}
where $\alpha,\beta,\gamma$ are real numbers and $q,q_1,q_2$ are the fractional orders of the Caputo derivatives, with $0<q_1<q_2<q\leq 2$. 

The linear homogeneous equation \eqref{eq.multi.order} can be thought of as the linearization at the trivial equilibrium of a nonlinear autonomous fractional-order differential equation of the form
\begin{equation}\label{eq.nonlin.multi}
   ^c\!D^q x(t)=f(x(t),^c\!D^{q_1} x(t),^c\!D^{q_2} x(t)) 
\end{equation}
where $f:\mathbb{R}^3\rightarrow\mathbb{R}$ is a continuously differentiable function such that $f(0,0,0)=0$. At the same time, the equation \eqref{eq.multi.order} is also strongly connected to the generic equation of a nonlinear damped pendulum with fractional damping terms
\[Ax''(t)+\sum_{k=1}^N B_k~^c\!D^{q_k} x(t)=f(t,x(t))\]
which has been investigated in \cite{brestovanska2014asymptotic,seredynska2005nonlinear}, pointing out that fractional derivatives are particularly useful in the mathematical modeling of  damping forces in vibrating systems in
viscous fluids.  

The initial condition associated to equation \eqref{eq.multi.order} is either
    \begin{enumerate}
        \item[i.] $x_i=x(0)=x_0\in\mathbb{R}$, if $q\in (0,1]$, or
        \item[ii.] $x_i=(x(0),x'(0))=(x_0,x_1)\in\mathbb{R}^2$ if $q\in(1,2)$.
    \end{enumerate}

In what follows, $\varphi(t,x_i)$ denotes the unique solution of (\ref{eq.multi.order}) satisfying one of the initial considered above. The existence and uniqueness of the solution of the initial value problem associated to the equation (\ref{eq.multi.order}) can be proved similarly to the case of fractional-order systems of differential equations \cite{Diethelm_book,seredynska2005nonlinear}. In fact, the solution of the linear constant coefficient fractional differential
equation \eqref{eq.multi.order} can be expressed using the fractional Green’s function as shown in \cite{Podlubny}, or by employing the fractional meta-trigonometric approach in the commensurate case \cite{lorenzo2013application}.

Due to the presence of the memory effect and hereditary properties, it is important to mention that the asymptotic stability of the trivial solution of equation (\ref{eq.multi.order}) is not of exponential type \cite{cermak2015stability,Gorenflo_Mainardi}. Instead, a non-exponential asymptotic stability concept is required in this case, called Mittag-Leffler stability \cite{Li_Chen_Podlubny}. Hereafter, in our work, we focus our attention on  $\mathcal{O}(t^{-\alpha})$-asymptotic stability, which reflects the algebraic decay of the solution.

\begin{definition}\label{def.stability}$ $
    \begin{enumerate}
        \item[i.] The trivial solution of (\ref{eq.multi.order}) is called \emph{stable} if for any $\varepsilon>0$
	there exists $\delta=\delta(\varepsilon)>0$ such that for every $x_i$ satisfying $\|x_i\|<\delta$ we have
	$|\varphi(t,x_i)|\leq\varepsilon$ for any $t\geq 0$.
	    \item[ii.] The trivial solution  of (\ref{eq.multi.order}) is called \emph{asymptotically stable} if it is stable and there
	exists $\rho>0$ such that $\lim\limits_{t\rightarrow\infty}\varphi(t,x_i)=0$ whenever $\|x_i\|<\rho$.
	    \item[iii.] Let $\alpha>0$. The trivial solution  of (\ref{eq.multi.order}) is called \emph{$\mathcal{O}(t^{-\alpha})$-asymptotically stable} if it is stable and there exists $\rho>0$ such that for any $\|x_i\|<\rho$ one has:
	$$|\varphi(t,x_i)|=\mathcal{O}(t^{-\alpha})\quad\textrm{as }t\rightarrow\infty.$$
	\end{enumerate}
\end{definition}

We recall the Laplace transform of the Caputo derivative or an arbitrary fractional order $q$ \cite{Podlubny}:
\begin{definition}
    The Laplace transform for the fractional-order Caputo derivative of order $q\in (n-1,n]$, $n\in\mathbb{N}^*$, of a function $x$ is: 
    $$\mathcal{L}(^c\!D^q x)(s)=s^qX(s)-\sum\limits_{k=0}^{n-1}s^{q-k-1}x^{(k)}(0),$$ where $X(s)$ represents the Laplace transform of the function $x$.
\end{definition}

\begin{rem}$ $
    \begin{enumerate}
        \item[i.] If $q\in (0,1]$, the Laplace transform of the fractional-order Caputo derivative is $$\mathcal{L}(^c\!D^q x)(s)=s^qX(s)-s^{q-1}x_0\quad \text{where}\quad x_0=x(0).$$
        \item[ii.] If $q\in(1,2]$, the Laplace transform of the fractional-order Caputo derivative is $$\mathcal{L}(^c\!D^q x)(s)=s^qX(s)-s^{q-1}x_0-s^{q-2}x_1\quad \text{where}\quad x_0=x(0)\ \text{and}\ x_1=x'(0).$$
    \end{enumerate}
\end{rem}

Applying the Laplace transform, equation \eqref{eq.multi.order} becomes
$$(s^q+\alpha s^{q_1}+\beta s^{q_2}+\gamma)X(s)=F(s),$$ where $X(s)$ represents the Laplace transform of the function $x$, $s^{q^*}$ represents the first branch of the complex power function \cite{Doetsch}, with $q^*\in \{q,q_1,q_2\}$ and $$F(s)=\begin{cases}
s^{q-1}x_0+\alpha s^{q_1-1}x_0+\beta s^{q_2-1}x_0,\ \text{if}\ 0<q_1<q_2<q<1\\
s^{q-1}x_0+s^{q-2}x_1+\alpha s^{q_1-1}x_0+\beta s^{q_2-1}x_0,\ \text{if}\ 0<q_1<q_2<1<q<2\\
s^{q-1}x_0+s^{q-2}x_1+\alpha s^{q_1-1}x_0+\beta s^{q_2-1}x_0+\beta s^{q_2-2}x_1,\ \text{if}\ 0<q_1<1<q_2<q<2\\
s^{q-1}x_0+s^{q-2}x_1+\alpha s^{q_1-1}x_0+\alpha s^{q_1-2}x_1+\beta s^{q_2-1}x_0+\beta s^{q_2-2}x_1,\ \text{if}\ 1<q_1<q_2<q<2
\end{cases}$$

Thus, we obtain the characteristic equation  
\begin{equation}\label{eq.char}
    s^q+\alpha s^{q_1}+\beta s^{q_2}+\gamma=0.
\end{equation}

The main contribution of this work is the analysis of the distribution of roots of equation \eqref{eq.char}. Therefore, we consider the complex valued function 
$$\Delta(s;\alpha,\beta,\gamma,q_1,q_2,q)=s^q+\alpha s^{q_1}+\beta s^{q_2}+\gamma.$$

For the particular case $q=q_1+q_2$, the characteristic equation \eqref{eq.char} becomes the characteristic equation corresponding to linear autonomous two-dimensional incommensurate fractional-order differential systems \cite{Brandibur_2018}. Therefore, following a similar proof as in \cite{Brandibur_2018}, we obtain an analogous result for the characterisation of stability and instability properties of equation \eqref{eq.multi.order}, in term of roots of the characteristic function $\Delta(s;\alpha,\beta,\gamma,q_1,q_2,q)$.  

\begin{theorem}$ $
    \begin{enumerate}
        \item[i.] Equation \eqref{eq.multi.order} is $\mathcal{O}(t^{-q'})$-asymptotically stable if and only if all the roots of the characteristic function $\Delta(s;\alpha,\beta,\gamma,q_1,q_2,q)$ are in the open left half-plane ($\Re(s)<0$), where $q'=\min\{\{q_1\},\{q_2\},\{q\}\}$, with $\{q^*\}=q^*-\lfloor q^* \rfloor$, $q^*\in \{q_1,q_2,q\}$.  
        \item[ii.] If $\gamma\ne 0$ and the characteristic function $\Delta(s;\alpha,\beta,\gamma,q_1,q_2,q)$ has at least one root in the open right half-plane ($\Re(s)>0$), equation \eqref{eq.multi.order} is unstable.
    \end{enumerate}
\end{theorem}

\section{Fractional-order-dependent stability and instability results}
\label{sec:3}

In this section, we will assume that the fractional orders $q_1,q_2,q$ are arbitrarily
fixed inside the domain 
\[D=\{(q_1,q_2,q)\in\mathbb{R}^3~:~0<q_1<q_2<q\leq 2\}.\]
Moreover, as $\gamma<0$ implies that the equation \eqref{eq.multi.order} is unstable, for any choice of the fractional orders $q_1,q_2,q$ (as it will be shown in section 3.2.), we will further assume that $\gamma>0$.  

\begin{lemma}\label{lem.curve.gamma} Let $(q_1,q_2,q)\in D$ and $\gamma>0$ arbitrarily fixed. Consider the smooth parametric curve in the $(\alpha,\beta)$-plane defined by
	$$
	\Gamma\left(\gamma,q_1,q_2,q\right)~:\quad
	\begin{cases}
		\alpha=\gamma^{1-\frac{q_1}{q}} h\left(\omega,q_1,q_2,q\right)\\
        \beta=\gamma^{1-\frac{q_2}{q}} h\left(\omega,q_2,q_1,q\right)
	\end{cases},\quad \omega>0,
	$$
	where $h:(0,\infty)\times D\rightarrow(0,\infty)$ is given by:
    $$h\left(\omega,q_1,q_2,q\right)=\omega^{-\frac{q_1}{q}}\left[\omega \rho(q-q_2,q_2-q_1)-\rho(q_2,q_2-q_1)\right]$$
with the function $\rho$ defined  as
$$\rho(a,b)=\frac{\sin\frac{a\pi}{2}}{\sin\frac{b\pi}{2}}\quad,~\forall~a\in[0,2],~b\in[-1,0)\cup(0,1].$$
The following statements hold:
\begin{itemize}	
\item[i.] The curve $\Gamma(\gamma,q_1,q_2,q)$ is the graph of a smooth, decreasing, convex bijective function $\phi_{\gamma,q_1,q_2,q}:\mathbb{R}\rightarrow\mathbb{R}$ in the $(\alpha,\beta)$-plane.
\item[ii.] The curve $\Gamma(\gamma,q_1,q_2,q)$ lies outside the first quadrant of the $(\alpha,\beta)$-plane.
\end{itemize}
\end{lemma}

\begin{proof}$ $ 
\emph{Proof of statement (i).} 
For simplicity, we denote
\begin{align*}
    & \rho_1=\rho(q_1,q_2-q_1),\quad\Tilde{\rho_1}=\rho(q-q_1,q_2-q_1)\\
    & \rho_2=\rho(q_2,q_2-q_1),\quad\Tilde{\rho_2}=\rho(q-q_2,q_2-q_1)
\end{align*}
A simple computation shows that \begin{align*}
    & \dfrac{\partial h}{\partial \omega}(\omega,q_1,q_2,q)=\dfrac{\omega^{-\frac{q_1}{q}-1}}{q}\left[ \omega\left(q-q_1\right)\tilde{\rho_2}+q_1\rho_2 \right]\\
    & \dfrac{\partial h}{\partial \omega}(\omega,q_2,q_1,q)=-\dfrac{\omega^{-\frac{q_2}{q}-1}}{q} \left[ \omega\left(q-q_2\right)\tilde{\rho_1}+q_2\rho_1 \right]
\end{align*}
As $q-q_1>0$, $q-q_2>0$ and $\rho_1,\rho_2,\tilde{\rho_1},\tilde{\rho_2}>0$, we have that $\dfrac{\partial h}{\partial \omega}(\omega,q_1,q_2,q)>0$ and $\dfrac{\partial h}{\partial \omega}(\omega,q_2,q_1,q)<0$  for any $\omega>0$.

Therefore, the real-valued function $\omega \mapsto h(\omega,q_1,q_2,q)$ is bijective and monotonous on $\mathbb{R}$: increasing if $q_2<q_1$ and decreasing otherwise. As $q_2>q_1$, it follows that the curve $\Gamma(\gamma,q_1,q_2,q)$ is the graph of a smooth decreasing bijective function $\phi_{\gamma,q_1,q_2,q}$ in the $(\alpha,\beta)$-plane.

Furthermore, by a similar reasoning:
\begin{align*}
    & \dfrac{\partial \alpha}{\partial \omega}=\dfrac{\gamma^{1-\frac{q_1}{q}}\omega^{-\frac{q_1}{q}-1}}{q}[\omega(q-q_1)\tilde{\rho_2}+q_1\rho_2]>0 &\\
    & \dfrac{\partial \beta}{\partial \omega}=-\dfrac{\gamma^{1-\frac{q_2}{q}}\omega^{-\frac{q_2}{q}-1}}{q}[\omega(q-q_2)\tilde{\rho_1}+q_2\rho_1]<0 &
\end{align*}
and hence, it follows that  $\dfrac{\partial\beta}{\partial \alpha}=\dfrac{\partial \beta} {\partial \omega}\cdot\dfrac{\partial \omega}{\partial \alpha}<0$, which implies that the function $\phi_{\gamma,Q_1,Q_2,q}$ is decreasing.

For the second order derivatives, we have
\begin{align*}
 & \dfrac{\partial^2 \alpha}{\partial \omega^2}=\!-\dfrac{q_1\gamma^{1-\frac{q_1}{q}}\omega^{-\frac{q_1}{q}-2}}{q^2}\left[ \omega (q-q_1)\tilde{\rho_2}+(q+q_1)\rho_2 \right]<0 &\\
 & \dfrac{\partial^2 \beta}{\partial \omega^2}=\!\dfrac{q_2\gamma^{1-\frac{q_2}{q}}\omega^{-\frac{q_2}{q}-2}}{q^2}\left[ \omega (q-q_2)\tilde{\rho_1}+(q+q_2)\rho_2 \right]>0  &
\end{align*}
In order to study the convexity of the curve $\Gamma(\gamma,q_1,q_2.q)$, we determine the sign of $\dfrac{\partial^2\beta}{\partial \alpha^2}$, where
$$\dfrac{\partial^2\beta}{\partial \alpha^2}=\dfrac{\partial}{\partial \alpha}\left( \dfrac{\partial \beta}{\partial \alpha} \right)=\dfrac{\dfrac{\partial^2\beta}{\partial\omega^2}\cdot \dfrac{\partial\alpha}{\partial\omega}-\dfrac{\partial^2\alpha}{\partial \omega^2}\cdot \dfrac{\partial \beta}{\partial\omega}}{\left( \dfrac{\partial\alpha}{\partial \omega} \right)^3}$$
As $ \dfrac{\partial\alpha}{\partial \omega}>0$, the convexity of curve $\Gamma(\gamma,q_1,q_2.q)$ is determined by the sign of the numerator of the previous expression. Indeed, we obtain:
\begin{align*}
    &\dfrac{\partial^2\beta}{\partial\omega^2}\cdot \dfrac{\partial\alpha}{\partial\omega}-\dfrac{\partial^2\alpha}{\partial \omega^2}\cdot \dfrac{\partial \beta}{\partial\omega}=\dfrac{(q_2-q_1)\gamma^{2-\frac{q_1+q_2}{q}}\omega^{-3-\frac{q_1+q_2}{q}}}{q^3}\left[ \omega^2(q-q_1)(q-q_2)\tilde{\rho_1}\tilde{\rho_2}+q_1q_1\rho_1\rho_2 \right]+\\
    &~~~+\dfrac{\gamma^{2-\frac{q_1+q_2}{q}}\omega^{-2-\frac{q_1+q_2}{q}}}{q^3}\left[ q_2(q_2-q_1)(q_2-q_1+q)\rho_1\tilde{\rho_2}+q_1(q-q_2)(q_2-q_1-q)\tilde{\rho_1}\rho_2 \right]
\end{align*}
%As $\gamma>0$, $\omega>0$, it follows that $\gamma^{2-\frac{q_1+q_2}{q}}>0$, $\omega^{-2-\frac{q_1+q_2}{q}}>0$ and $\omega^{-3-\frac{q_1+q_2}{q}}>0$. Moreover, $q_2-q_1>0$, $q-q_1>0$, $q-q_2>0$, $q_1q_2>0$ and $\rho_1,\rho_2,\tilde{\rho_1},\tilde{\rho_2}>0$. Therefore, the first term of the previous relation is positive.
It is easy to see that the first addend of this sum is positive. On the other hand, for the second addend, we obtain
\begin{align*}
    &q_2(q-q_1)(q_2-q_1+q)\rho_1\tilde{\rho_2}+q_1(q-q_2)(q_2-q_1-q)\tilde{\rho_1}\rho_2=\\
    &~~~~~=q_1q_2(q-q_1)(q-q_2)\left[ (q_2-q_1)\left( \dfrac{\rho_1\tilde{\rho_2}}{q_1(q-q_2)}+\dfrac{\tilde{\rho_1}\rho_2}{q_2(q-q_1)} \right)-q\left( \dfrac{\rho_1\tilde{\rho_2}}{q_1(q-q_2)}-\dfrac{\tilde{\rho_1}\rho_2}{q_2(q-q_1)} \right) \right]
\end{align*}
It is clear that $\dfrac{\rho_1\tilde{\rho_2}}{q_1(q-q_2)}+\dfrac{\tilde{\rho_1}\rho_2}{q_2(q-q_1)}>0$.

Moreover, as the function $x\mapsto\dfrac{\sin x}{x}$ is decreasing on $(0,\pi)$ and $q_1<q_2$, we have that $q-q_1>q-q_2$ and hence
$$\dfrac{\sin \frac{q_1\pi}{2}}{\frac{q_1\pi}{2}}>\dfrac{\sin\frac{q_2\pi}{2}}{\frac{q_2\pi}{2}}\quad \text{and}\quad \dfrac{\sin \frac{(q-q_2)\pi}{2}}{\frac{(q-q_2)\pi}{2}}>\dfrac{\sin \frac{(q-q_1)\pi}{2}}{\frac{(q-q_1)\pi}{2}}$$
Multiplying the previous two inequalities, it follows that
$$\dfrac{\sin \frac{q_1\pi}{2}}{\frac{q_1\pi}{2}}\cdot \dfrac{\sin \frac{(q-q_2)\pi}{2}}{\frac{(q-q_2)\pi}{2}}>\dfrac{\sin \frac{q_2\pi}{2}}{\frac{q_2\pi}{2}}\cdot \dfrac{\sin \frac{(q-q_1)\pi}{2}}{\frac{(q-q_1)\pi}{2}}.$$
which leads to $$\dfrac{\tilde{\rho_2}}{q-q_2}\cdot \dfrac{\rho_1}{q_1}-\dfrac{\rho_2}{q_2}\cdot \dfrac{\tilde{\rho_1}}{q-q_1}>0.$$
It finally results that
$$q_2(q-q_1)(q_2-q_1+q)\rho_1\tilde{\rho_2}+q_1(q-q_2)(q_2-q_1-q)\tilde{\rho_1}\rho_2>0,$$
and therefore, $\dfrac{\partial^2\beta}{\partial\alpha^2}>0$, meaning that the function  $\phi_{\gamma,q_1,q_2,q}$ is convex.

\emph{Proof of statement (ii).} Assuming the contrary, i.e. there exists $\omega>0$ such that $\alpha>0$ and $\beta>0$, we obtain
$$\begin{cases} 
& \omega \rho(q-q_2,q_2-q_1)>\rho(q_2.q_2-q_1)\\
& \omega \rho(q-q_1,q_1-q_2)>\rho(q_1.q_1-q_2)
\end{cases}$$
which implies
\begin{equation}\label{eq.trig}
    \dfrac{\sin (q-q_2)\frac{\pi}{2}}{\sin (q-q_1)\frac{\pi}{2}}>\dfrac{\sin q_2\frac{\pi}{2}}{\sin q_1\frac{\pi}{2}}.
\end{equation}
Using elementary trigonometric identities, it follows that inequality \eqref{eq.trig} is equivalent to the following inequality $$2\sin \frac{q\pi}{2}\sin(q_2-q_1)\frac{\pi}{2}<0,$$ which is absurd as $\sin \frac{q\pi}{2}>0$ and $\sin (q_2-q_1)\frac{\pi}{2}>0$.

Hence, the curve $\Gamma(\gamma,q_1,q_2,q)$ does not have any points in the first quadrant of the $(\alpha,\beta)$-plane.
\end{proof}

\begin{rem}
    If $q_1=q_2=:q^*$, $\Gamma(\gamma,q_1,q_2,q)$ represents the line
    $$\alpha+\beta=-\dfrac{\gamma^{1-\frac{q^*}{q}}\sin\dfrac{q\pi}{2}}{\left(\sin\dfrac{q^*\pi}{2}\right)^{\frac{q^*}{q}}\left(\sin\dfrac{(q-q^*)\pi}{2}\right)^{1-\frac{q^*}{q}}}.$$
\end{rem}

In the following, we will denote by $N(\alpha,\beta,\gamma,q_1,q_2,q)$ the number of unstable roots ($\Re(s)\geq 0$) of the characteristic function $\Delta(s;\alpha,\beta,\gamma,q_1,q_2,q)$, including their multiplicities. The next lemma is concerned with the well-definedness of the function $N(\alpha,\beta,\gamma,q_1,q_2,q)$, as well as several of its properties which will be used to obtain the main results of this paper. 

\begin{lemma}\label{lemma.number.roots} Let $\gamma>0$ and consider arbitrarily fixed fractional orders $0<q_1<q_2<q\leq 2$.
\begin{itemize}
\item[i.] There exist a strictly decreasing function $l_{\gamma,q_1,q_2,q}:\mathbb{R}^+\rightarrow\mathbb{R}^+$ and a strictly increasing function $L_{\gamma,q_1,q_2,q}:\mathbb{R}^+\rightarrow\mathbb{R}^+$ such that any unstable root of $\Delta(s;\alpha,\beta,\gamma,q_1,q_2,q)$ is bounded by
\begin{equation}\label{ineq.s}
    l_{\gamma,q_1,q_2,q}(|\alpha|+|\beta|)\leq |s|\leq L_{\gamma,q_1,q_2,q}(|\alpha|+|\beta|).
\end{equation}

\item[ii.] For the characteristic function $\Delta(s;\alpha,\beta,\gamma,q_1,q_2,q)$ there exist at most a finite number of roots such that $\Re(s)\geq 0$. 

\item[iii.]  The function $(\alpha,\beta)\mapsto N(\alpha,\beta,\gamma,q_1,q_2,q)$ is continuous at all points $(\alpha,\beta)$ that do not belong to the curve $\Gamma(\gamma,q_1,q_2,q)$. Therefore, $N(\alpha,\beta,\gamma,q_1,q_2,q)$ is constant on each connected component of $\mathbb{R}^2\setminus \Gamma(\gamma,q_1,q_2,q)$.
\end{itemize}
\end{lemma}

\begin{proof}
\emph{Proof of statement (i).} 
Considering the function $f(x;A,B,\varepsilon)=x-Ax^{\varepsilon}-B$, with $A,B>0$, $\varepsilon\in (0,1)$, and denoting by $x^*(A,B,\varepsilon)$ the unique positive root of the equation $f(x,A,B,\varepsilon)=0$, we can easily prove that $\dfrac{\partial x^*}{\partial A}>0$.
Moreover, as $f(\infty)=\infty$, we have that $f(x)>0$, for any $x>x^*$, where the arguments were dropped for simplicity.

Let $s$ be a root of the characteristic function $\Delta(s;\alpha,\beta,\gamma,q_1,q_2,q)$, i.e. 
\[s^q+\alpha s^{q_1}+\beta s^{q_2}+\gamma=0.\]

If $|s|>1$, then we obtain 

$$|s|^q\leq |\alpha||s|^{q_1}+|\beta||s|^{q_2}+\gamma\leq (|\alpha|+|\beta|)|s|^{q_2}+\gamma.$$
Denoting $|s|^q=\eta>1$ and substituting in the previous inequality, we have: $$\eta\leq (|\alpha|+|\beta|)\eta^{\frac{q_2}{q}}+\gamma,$$
providing that $$f\left(\eta;|\alpha|+|\beta|,\gamma,\dfrac{q_2}{q}\right)\leq 0,$$ which is equivalent to $$\rho\leq x^*\left(|\alpha|+|\beta|,\gamma,\dfrac{q_2}{q}\right)<x^*\left(|\alpha|+|\beta|,\gamma,\dfrac{q_2}{q}\right)+1.$$
Therefore, $|s|<\left( x^*\left(|\alpha|+|\beta|,\gamma,\dfrac{q_2}{q}\right)+1 \right)^{\frac{1}{q}}$.

On the other hand, if $|s|\leq 1$, we have: 
$$\gamma |s|^{-q}\leq 1+|\alpha||s|^{q_1-q}+|\beta||s|^{q_2-q}.$$

Denoting $|s|^{-q}=\eta$, it results that $$\eta\leq \dfrac{1}{\gamma}+\dfrac{|\alpha|+|\beta|}{\gamma}\eta^{1-\frac{q_1}{q}},$$
and hence: $$f\left( \eta;\dfrac{|\alpha|+|\beta|}{\gamma},\dfrac{1}{\gamma},1-\dfrac{q_1}{q} \right)\leq 1,$$ which is equivalent to $$\eta\leq x^*\left( \dfrac{|\alpha|+|\beta|}{\gamma},\dfrac{1}{\gamma},1-\dfrac{q_1}{q} \right)<x^*\left( \dfrac{|\alpha|+|\beta|}{\gamma},\dfrac{1}{\gamma},1-\dfrac{q_1}{q} \right)+1.$$
Thus, $|s|\geq \left( x^*\left( \dfrac{|\alpha|+|\beta|}{\gamma},\dfrac{1}{\gamma},1-\dfrac{q_1}{q} \right)+1 \right)^{-\frac{1}{q}}$.

In conclusion, chosing the strictly decreasing function $l_{\gamma,q_1,q_2,q}:\mathbb{R}^+\to\mathbb{R}^+$ given by $$l_{\gamma,q_1,q_2,q}(|\alpha|+|\beta|)=\left( x^*\left( \dfrac{|\alpha|+|\beta|}{\gamma},\dfrac{1}{\gamma},1-\dfrac{q_1}{q} \right)+1 \right)^{-\frac{1}{q}}$$ and the strictly increasing function $L_{\gamma,q_1,q_2,q}:\mathbb{R}^+\to\mathbb{R}^+$ given by $$l_{\gamma,q_1,q_2,q}(|\alpha|+|\beta|)=\left( x^*\left(|\alpha|+|\beta|,\gamma,\dfrac{q_2}{q}\right)+1 \right)^{\frac{1}{q}},$$ we obtain inequality \eqref{ineq.s}.

\emph{Proof of statement (ii).} By \textit{reductio ad absurdum}, if the characteristic function $\Delta(s;\alpha,\beta,\gamma,q_1,q_2,q)$ has an infinite number of unstable roots, by the Bolzano-Weierstrass theorem, we can find a convergent sequence of unstable roots $(s_n)$, such that $s_n\rightarrow s_0$ when $n\rightarrow\infty$. As $\gamma>0$, it follows that $s_0\neq 0$ and $\Re(s_0)\geq 0.$ The function $\Delta(s;\alpha,\beta,\gamma,q_1,q_2,q)$ is analytic on $\mathbb{C}\setminus \mathbb{R}^{-}$ and, from the principle of permanence, we have that it is identically zero, which contradicts our claim. Hence, $N(\alpha,\beta,\gamma,q_1,q_2,q)$ is finite.

\emph{Proof of statement (iii).} Let $v_0=(\alpha_0,\beta_0)\in\mathbb{R}^2\setminus \Gamma (\gamma,q_1,q_2,q)$ and $r>0$ such that the open neighborhood of the point $a_0$ $$B_r(v_0)=\{ v=(\alpha,\beta)\in\mathbb{R}^2\ : \ \|v-v_0\|_1<r \}$$ is included in the region $\mathbb{R}^2\setminus \Gamma (\gamma,q_1,q_2,q)$, where $\|v\|_1=|\alpha|+|\beta|$.

Then, for any point $v=(\alpha,\beta)\in B_r(v_0)$, we have $$\|v\|_1\leq \|v-v_0\|_1+\|v_0\|_1<r+\|v_0\|_1,$$ an therefore, any root $s$ of the characteristic function $\Delta(s;\alpha,\beta,\gamma,q_1,q_2,q)$ with $\Re(s)\geq 0$ satisfies the inequality $$l_{\gamma,q_1,q_2,q}(r+\|v_0\|_1)<|s|<L_{\gamma,q_1,q_2,q}(r+\|v_0\|_1).$$
We denote $l=l_{\gamma,q_1,q_2,q}(r+\|v_0\|_1)$ and $L=L_{\gamma,q_1,q_2,q}(r+\|v_0\|_1)$ and consider the closed curve $(C)$ in the complex plane, which is oriented counterclockwise and bounds the open set $$R=\{s\in\mathbb{C}\ :\ \Re(s)>0,\ l<|s|<L \}.$$

Therefore, for any point $v=(\alpha,\beta)\in B_r(v_0)$, all unstable roots of the function $\Delta(s;\alpha,\beta,\gamma,q_1,q_2,q)$ (i.e. $\Re(s)\geq 0$) are inside the set $R$.

Because $\Delta(s;\alpha_0,\beta_0,\gamma,q_1,q_2,q)\ne 0$, for any $s$ on the considered closed curve, we can easily see that $$d_0=\min\limits_{s\in (C)}|\Delta(s;\alpha_0,\beta_0,\gamma,q_1,q_2,q)|>0.$$

Next, we consider $$r'=\min \left \{ r,\dfrac{d_0}{\|(L^{q_1},L^{q_2})\|_{\infty}} \right \}.$$

Then, for any $s\in (C)$ and $a\in B_{r'}(v_0)\subset B_r(v_0)$, by applying H\"{o}lder's inequality, it follows that

\begin{align*}
    |\Delta & (s;\alpha,\beta,\gamma,q_1,q_2,q)-\Delta(s;\alpha_0,\beta_0,\gamma,q_1,q_2,q)|=\\
    &=|(\alpha-\alpha_0)s^{q_1}+(\beta-\beta_0)s^{q_2}|\leq\\
    &\leq|\alpha-\alpha_0|L^{q_1}+|\beta-\beta_0|L^{q_2}\leq \\
    &\leq \|v-v_0\|_1\cdot \|(L^{q_1},L^{q_2})\|_{\infty}<\\
    &<r'\|(L^{q_1},L^{q_2})\|_{\infty}\leq d_0\leq |\Delta(s;\alpha_0,\beta_0,\gamma,q_1,q_2,q)|
\end{align*}

Therefore, by Rouch\'{e}'s theorem, we have that the functions $\Delta(s;\alpha,\beta,\gamma,q_1,q_2,q)$ and $\Delta(s;\alpha_0,\beta_0,\gamma,q_1,q_2,q)$ have the same number of roots in the set $R$.

Thus, $N(\alpha,\beta,\gamma,q_1,q_2,q)=N(\alpha_0,\beta_0,\gamma,q_1,q_2,q)$, for any $a\in B_{r'}(v_0)$ and therefore, the function $(\alpha,\beta)\mapsto N(\alpha,\beta,\gamma,q_1,q_2,q)$ is continuous on the domain $\mathbb{R}^2\setminus \Gamma(\gamma,q_1,q_2,q)$ and, because it is an integer-valued function, we obtain that it takes constant values on every connected component of $\mathbb{R}^2\setminus \Gamma(\gamma,q_1,q_2,q)$.
\end{proof}

With these preliminary lemmas, we now present the main result of this section. 

\begin{theorem}[Fractional-order-dependent stability and instability results]\label{thm.q}$ $

Let $\gamma>0$, $0<q_1<q_2< q\leq 2$ arbitrarily fixed. Consider the curve $\Gamma(\gamma,q_1,q_2,q)$ and the function $\phi_{\gamma,q_1,q_2,q}:\mathbb{R}\rightarrow\mathbb{R}$ defined in Lemma \ref{lem.curve.gamma}.
\begin{itemize}
\item[i.] The characteristic equation \eqref{eq.char} has a pair of complex conjugated roots on the imaginary axis of the complex plane if and only if $(\alpha,\beta)\in \Gamma(\gamma;q_1,q_2,q)$.
			
\item[ii.] The trivial solution of equation (\ref{eq.multi.order}) is $\mathcal{O}(t^{-q'})$-asymptotically stable (where $q'=\min\{\{q\},\{q_1\},\{q_2\}\}$) if and only if $$\beta>\phi_{\gamma,q_1,q_2,q}(\alpha).$$

\item[iii.] If $\beta<\phi_{\gamma,q_1,q_2,q}(\alpha)$, the trivial solution of equation (\ref{eq.multi.order}) is unstable.
\end{itemize}
\end{theorem}

\begin{proof} 
\emph{Proof of statement (i).} 
    The characteristic equation \eqref{eq.char} has a pair of complex conjugated roots on the imaginary axis of the complex plane if and only if there exists $\omega>0$ such that $$\Delta(i(\gamma \omega)^{\frac{1}{q}};\alpha,\beta,\gamma,q_1,q_2,q)=0.$$
    Taking the real and imaginary part of the previous relation, we obtain
    $$\begin{cases}
    & \omega \cos \dfrac{q\pi}{2}+\alpha \omega^{\frac{q_1}{q}}\gamma^{\frac{q_1}{q}-1}\cos \dfrac{q_1\pi}{2}+\beta\omega^{\frac{q_2}{q}}\gamma^{\frac{q_2}{q}-1}\cos \dfrac{q_2\pi}{2}+1=0\\\\
    & \omega \sin \dfrac{q\pi}{2}+\alpha \omega^{\frac{q_1}{q}}\gamma^{\frac{q_1}{q}-1}\sin \dfrac{q_1\pi}{2}+\beta\omega^{\frac{q_2}{q}}\gamma^{\frac{q_2}{q}-1}\sin \dfrac{q_2\pi}{2}=0
    \end{cases}$$
    Solving the above system for $\alpha$ and $\beta$, it follows that the characteristic equation \eqref{eq.char} has a pair of pure imaginary roots if and only if the pair of parameters $(\alpha,\beta)$ belongs to the curve $\Gamma(\gamma,q_1,q_2,q)$ defined in Lemma \ref{lem.curve.gamma}.
    
\emph{Proof of statement (ii).} 
    Choosing $\alpha=\beta=1$, we will first prove that the characteristic function $\Delta(s;1,1,\gamma,q_1,q_2,q)$ does not have any roots with positive real part.
    
    Assuming the contrary, i.e. that the function $\Delta(s;1,1,\gamma,q_1,q_2,q)$ admits a root $s\in \mathbb{C}$ such that $\Re(s)\geq 0$, it is clear that   $$s^q+s^{q_1}+s^{q_2}+\gamma=0.$$
    As $\gamma>0$, it is obvious that $s\neq 0$. Dividing the previous equation by $s^{\frac{q}{2}}$, it follows that $$s^{\frac{q}{2}}+s^{q_1-\frac{q}{2}}+s^{q_2-\frac{q_2}{2}}+\gamma s^{-\frac{q}{2}}=0.$$
    As $\pm \frac{q}{2}$, $q_1-\frac{q}{2}$, $q_2-\frac{q}{2}\in [-1,1]$, it results that the real part of each term of this equation is positive, which leads to a contradiction.
    
    Therefore, we obtain that $N(1,1,\gamma,q_1,q_2,q)=0$, and based on Lemma \ref{lemma.number.roots}, it follows that $N(\alpha,\beta,\gamma,q_1,q_2,q)=0$ for any $(\alpha,\beta)$ belonging to the connected component of the set $\mathbb{R}^2\setminus \Gamma(\gamma,q_1,q_2,q)$ which contains the point $(1,1)$. Hence, we obtain the desired conclusion.
    
\emph{Proof of statement (iii).} 
    Consider the terms $\rho_1,\rho_2,\tilde{\rho_1},\tilde{\rho_2}$ introduced in the proof of Lemma \ref{lem.curve.gamma} (i).  Let $s(\alpha,\beta,\gamma,q_1,q_2,q)$ denote the root of the characteristic function $\Delta(s;\alpha,\beta,\gamma,q_1,q_2,q)$ satisfying the relation $$s(\alpha^*,\beta^*,\gamma,q_1,q_2,q)=i\delta,$$ where $\delta=(\gamma \omega)^{\frac{1}{q}}$ from the proof of statement $(i)$, with $(\alpha^*.\beta^*)\in \Gamma(\gamma;q_1,q_2,q)$.
    
    Differentiating with respect to $\alpha$ in the characteristic equation 
    $$s^q+\alpha s^{q_1}+\beta s^{q_2}+\gamma =0$$ leads to 
    $$qs^{q-1}\dfrac{\partial s}{\partial \alpha}+s^{q_1}+\alpha q_1s^{q_1-1}\dfrac{\partial s}{\partial \alpha}+\beta q_2s^{q_2-1}\dfrac{\partial s}{\partial \alpha}=0.$$
    Hence, taking the real part in this equation, we obtain 
    $$\dfrac{\partial \Re (s)}{\partial \alpha}=\Re \left( \dfrac{-s^{q_1}}{qs^{q-1}+\alpha q_1s^{q_1-1}+\beta q_2 s^{q_2-1}} \right).$$
    Moreover, we have 
    $$\dfrac{\partial \Re (s)}{\partial \alpha}{\Big|_{(\alpha^*,\beta^*)}}=\Re \left( \dfrac{-(i\delta)^{q_1}}{P(i\delta)} \right)=-\delta^{q_1}\Re \left( \dfrac{i^{q_1}\overline{P(i\delta)}}{|P(i\delta)|^2} \right)=\dfrac{-\delta^{q_1}}{|P(i\delta)|^2}\Re (i^{q_1}\overline{P(i\delta)}),$$ where $$P(s)=qs^{q-1}+\alpha^*q_1s^{q_1-1}+\beta^* q_2s^{q_2-1}.$$
        We further obtain
\begin{align*}
&\Re (i^{q_1}\overline{P(i\delta)})=\Re (\overline{i^{q_1}}P(i\delta))=\Re \left( q\gamma^{1-\frac{1}{q}}\omega^{-\frac{1}{q}}\left[i^{q-q_1-1}\omega-\frac{q_1}{q} i(\omega\tilde{\rho_2}-\rho_2)-\frac{q_2}{q} i^{q_2-q_1-1}(\omega\tilde{\rho_1}-\rho_1)\right]\right)
\end{align*}

As $\Re (i^{q-q_1-1})=\sin\frac{(q-q_1)\pi}{2}$ and $\Re (i^{q_2-q_1-1})=\sin\frac{(q_2-q_1)\pi}{2}$ and the second term of the previous expression is purely imaginary, we get
$$\Re (i^{q_1}\overline{P(i\delta)})=-q\gamma^{\frac{q_2-1}{q}}\omega^{1+\frac{q_2-1}{q}}\sin \dfrac{(q_2-q_1)\pi}{2}\cdot\dfrac{\partial \beta}{\partial\omega}.$$
It follows that 
$$\dfrac{\partial\Re(s)}{\partial\alpha}{\Big|_{(\alpha^*,\beta^*)}}=\dfrac{q\gamma^{\frac{q_1+q_2-1}{q}}\omega^{\frac{q_1+q_2-1}{q}}}{|P(i\delta)|^2}\sin\dfrac{(q_2-q_1)\pi}{2}\cdot \dfrac{\partial \beta}{\partial\omega}.$$
Along the same lines, we compute $\dfrac{\partial\Re(s)}{\partial\beta}\Big|_{(\alpha^*,\beta^*)}$, leading to the following gradient vector:
\begin{align*}\nabla &\Re(s)(\alpha^*,\beta^*)=\left(\frac{\partial\Re(s)}{\partial \alpha},\frac{\partial\Re(s)}{\partial \beta}\right)\bigg\rvert_{(\alpha^*,\beta^*)}=\\&=\dfrac{q\gamma^{\frac{q_1+q_2-1}{q}}\omega^{\frac{q_1+q_2+q-1}{q}}}{|P(i\delta)|^2}\sin\dfrac{(q_2-q_1)\pi}{2}\cdot \left( \dfrac{\partial\beta}{\partial\omega},-\dfrac{\partial \alpha}{\partial\omega} \right).
\end{align*} 
Moreover, taking into account the parametric equations of the curve $\Gamma(\gamma,q_1,q_2,q)$, we can easily see that 
the gradient vector $\nabla \Re(s)(\alpha^*,\beta^*)$ is a normal vector to the curve  $\Gamma(\gamma,q_1,q_2,q)$, pointing towards the region below the curve. Consequently, the following transversality condition holds for the directional derivative:
$$\nabla_{\overline{u}} \Re(s)(\alpha^*,\beta^*)=\left\langle\nabla \Re(s)(\alpha^*,\beta^*),\overline{u}\right\rangle>0,
$$
for any vector $\overline{u}$ pointing towards the region below the curve $\Gamma(\gamma,q_1,q_2,q)$. Therefore, as the parameters $(\alpha,\beta)$ vary and cross the curve $\Gamma(\gamma,q_1,q_2,q)$ into the region below the curve, $\Re(s)$ increases and becomes positive, i.e. the pair of complex conjugated roots $(s,\overline{s})$ crosses the imaginary axis from the open left half-plane to the open right half-plane. Hence, $N(\alpha,\beta,\gamma,q_1,q_2,q)= 2$ for any $(\alpha,\beta)$ belonging to the region below the curve $\Gamma(\gamma,q_1,q_2,q)$, and the equation
\eqref{eq.multi.order} is unstable.
\end{proof}

\section{Fractional-order-independent stability and instability results}
\label{sec:4}$ $

In this section, based on the results presented in the previous section, we seek to obtain fractional-order-independent stability and instability conditions for the fractional-order differential equation \eqref{eq.multi.order}, expressed in terms of the constant parameters $\alpha,\beta,\gamma$. 

With this aim in mind, the following lemma gives sufficient conditions for the instability of equation \eqref{eq.multi.order}, for any choice of the fractional orders $q_1,q_2,q$. 

\begin{lemma}\label{unstab.lem}
    If $\alpha+\beta+\gamma+1\leq 0$ or $\gamma<0$, the trivial solution of equation \eqref{eq.multi.order} is unstable, regardless of the fractional orders $q_1$, $q_2$ and $q$.
\end{lemma}

\begin{proof}
    Proving this result will resume to showing that the characteristic function $\Delta(s;\alpha,\beta,\gamma,q_1,q_2,q)$ has at least one positive real root.    
    
    On one hand, if $\alpha+\beta+\gamma+1\leq 0$ we can easily see that $$\Delta(1;\alpha,\beta,\gamma,q_1,q_2,q)=1+\alpha+\beta+\gamma\leq 0.$$
    
    Moreover, if $\gamma<0$, it is obvious that 
    $$\Delta(0;\alpha,\beta,\gamma,q_1,q_2,q)=\gamma< 0.$$
    
    On the other hand, we notice that $$\Delta(s;\alpha,\beta,\gamma,q_1,q_2,q)\rightarrow\infty\quad \text{when}\ s\to\infty.$$
    
    Hence, in both cases, the function $s\mapsto \Delta(s;\alpha,\beta,\gamma,q_1,q_2,q)$ has at least one strictly positive real root. Therefore, equation \eqref{eq.multi.order} is unstable, regardless of the fractional orders $q_1$, $q_2$ and $q$.
\end{proof}

A sufficient condition for the asymptotic stability of the equation \eqref{eq.multi.order}, for any choice of the fractional orders $(q_1,q_2,q)\in D$ is given by the following lemma. From now on, only the case $\gamma>0$ will be discussed, as Lemma \ref{unstab.lem} provides that if $\gamma<0$, the equation \eqref{eq.multi.order} is unstable, for any choice of the fractional orders $(q_1,q_2,q)\in D$.

\begin{lemma}\label{stab.lem}
    If $\alpha> 0$, $\beta> 0$ and $\gamma>0$, the trivial solution of equation \eqref{eq.multi.order} is asymptotically stable, regardless of the fractional orders $q_1$, $q_2$ and $q$.
\end{lemma}

\begin{proof}
    Let $\alpha> 0$, $\beta> 0$ and $\gamma>0$. By contradiction, we assume that $\Delta(s;\alpha,\beta,\gamma,q_1,q_2,q)$ has a root $s_0\neq 0$ with $\Re(s_0)\geq 0$. Then $$s_0^q+\alpha s_0^{q_1}+\beta s_0^{q_2}+\gamma=0.$$
    
    Then, multiplying this equality by $s_0^{-\frac{q}{2}}$, we obtain $$s_0^{\frac{q}{2}}+\alpha s_0^{q_1-\frac{q}{2}}+\beta s_0^{q_2-\frac{q}{2}}+\gamma s_0^{-\frac{q}{2}}=0.$$
    
    As $\pm\frac{q}{2}\in[-1,1]$ and $q_1-\frac{q}{2},q_2-\frac{q}{2}\in (-1,1)$, it follows that the real parts of each term from the left-hand side of this equation are positive and $$\Re \left(s_0^{\frac{q}{2}}+\alpha s_0^{q_1-\frac{q}{2}}+\beta s_0^{q_2-\frac{q}{2}}+\gamma s_0^{-\frac{q}{2}}\right)>0,$$ which is absurd.
    
    Hence, $\Re(s)<0$, for any root $s$ of the characteristic function $\Delta(s;\alpha,\beta,\gamma,q_1,q_2,q)$, which means that the trivial solution of equation \eqref{eq.multi.order} is asymptotically stable, regardless of the fractional orders $q$, $q_1$ and $q_2$.
\end{proof}

In the following, for $\gamma>0$, the following regions are considered in the $(\alpha,\beta)$-plane:
\begin{align*}
    S(\gamma)&=\{ (\alpha,\beta)\in\mathbb{R}^2~:~\alpha>0,\beta>0 \} \\
    U(\gamma)&=\{ (\alpha,\beta)\in\mathbb{R}^2~:~\alpha+\beta+\gamma+1\leq 0 \}.
\end{align*}

We have shown in Lemma \ref{unstab.lem} that if $(\alpha,\beta)\in U(\gamma)$, then the trivial solution of equation \eqref{eq.multi.order} is unstable stable, regardless of the fractional orders $q_1$, $q_2$ and $q$, whereas in Lemma \ref{stab.lem}, if $(\alpha,\beta)\in S(\gamma)$, the trivial solution of \eqref{eq.multi.order} is asymptotically stable, regardless of the fractional orders $q_1$, $q_2$ and $q$.  Furthermore, we will show that the conditions provided by the sufficiency Lemmas \ref{unstab.lem} and \ref{stab.lem} are at the same time necessary conditions for the fractional-order-independent asymptotic stability / instability of equation \eqref{eq.multi.order}, respectively. Therefore, we will call the regions $S(\gamma)$ and $U(\gamma)$ \textit{fractional-order-independent asymptotic stability / instability regions} for the equation \eqref{eq.multi.order}, respectively.

\begin{rem}
In Fig. \ref{fig.curves.frac}, we have plotted several curves $\Gamma(\gamma,q_1,q_2,q)$ for $\gamma=4$ and different values of fractional orders $q_1,q_2,q$ such that $0<q_1<q_2<q\leq 2$, as well as the fractional-order-independent stability and instability regions $S(\gamma)$ and $U(\gamma)$, respectively (plotted in darker shades). The regions below and above each curve (plotted in lighter shades) represent the  fractional-order-dependent asymptotic stability and instability regions, respectively, corresponding to the particular values of the fractional orders $q_1,q_2,q$ chosen for each plot.
\end{rem}

\begin{figure}[htbp]
    \centering
    \includegraphics[width=0.8\linewidth]{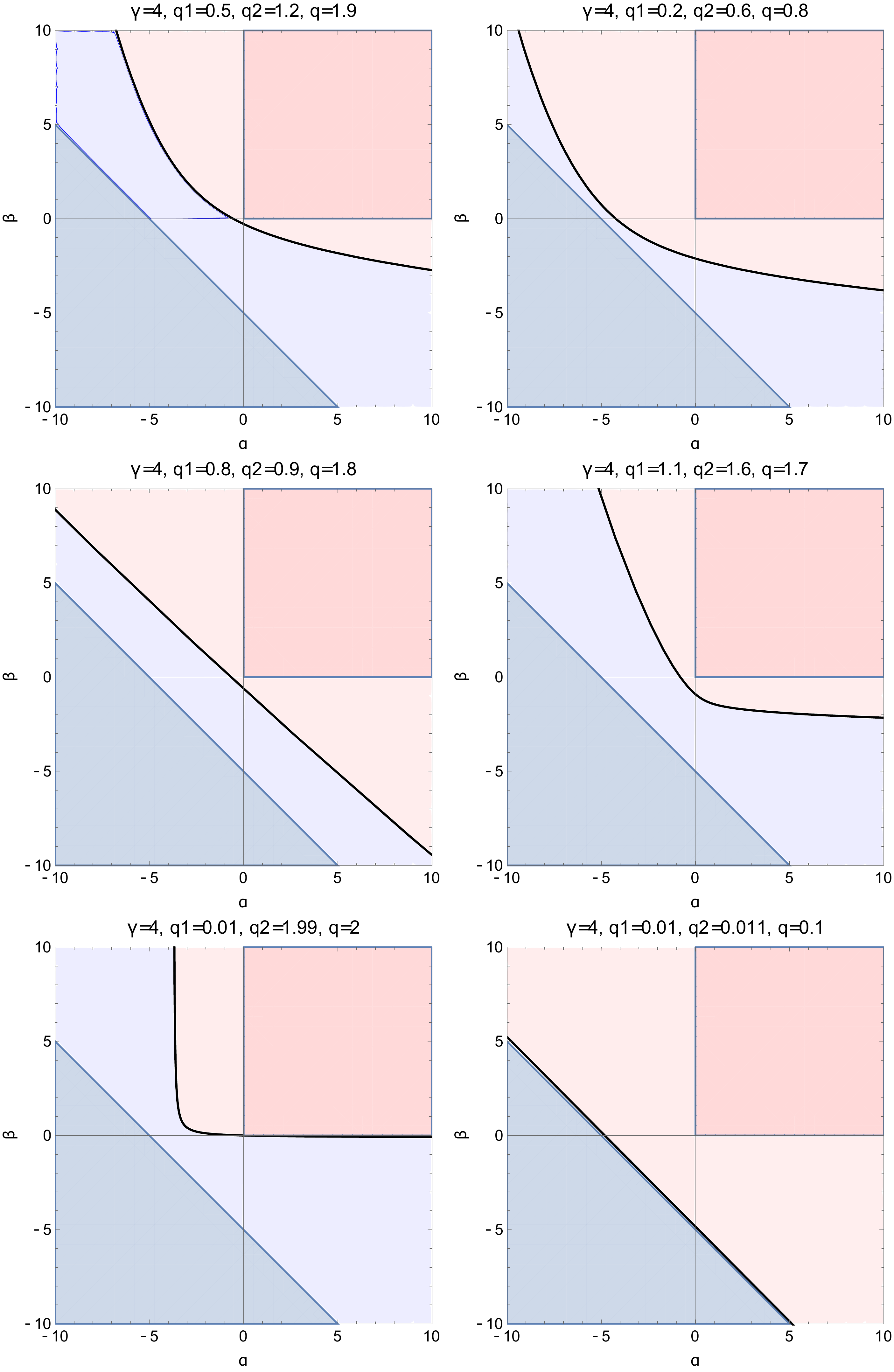}
    \caption{Curves $\Gamma(\gamma,q_1,q_2,q)$ (black) for $\gamma=4$ and different values of the fractional orders such that $0<q_1<q_2<q<2$. Darker shades of red / blue represent the fractional-order independent asymptotic stability / instability regions $S(\gamma)$ and $U(\gamma)$. Lighter shades of red / blue are associated to fractional-order-dependent stability/instability regions.}
    \label{fig.curves.frac}
\end{figure}

In the following, $Q(\gamma)$ denotes the region of the $(\alpha,\beta)$-plane which is determined by the curves $\Gamma(\gamma,q_1,q_2,q)$ given in Lemma \ref{lem.curve.gamma}:
$$Q(\gamma)=\{ (\alpha,\beta)\in\mathbb{R}^2~:~\exists~ 0<q_1<q_2<q\leq 2 ~ \text{s.t.}~ (\alpha,\beta)\in\Gamma(\gamma,q_1,q_2,q)  \}.$$

\begin{rem}
    In Fig. \ref{fig.curves.all}, a large number of curves belonging to region $Q(\gamma)$ is plotted, for different values of randomly generated fractional orders $(q_1,q_2,q)\in D$ (a total of $11^3$ triplets). It can be noticed that the plotted curves do not intersect neither the stability region $S(\gamma)$, nor the instability region $U(\gamma)$. Furthermore, the next result will actually prove that the union of all the curves $\Gamma(\gamma,q_1,q_2,q)$ (i.e. the set $Q(\gamma)$ defined above) fills the whole region of the $(\alpha,\beta)$-plane which separates the sets $S(\gamma)$ and $U(\gamma)$.
\end{rem}

\begin{figure}[htbp]
    \centering
    \includegraphics[width=0.5\linewidth]{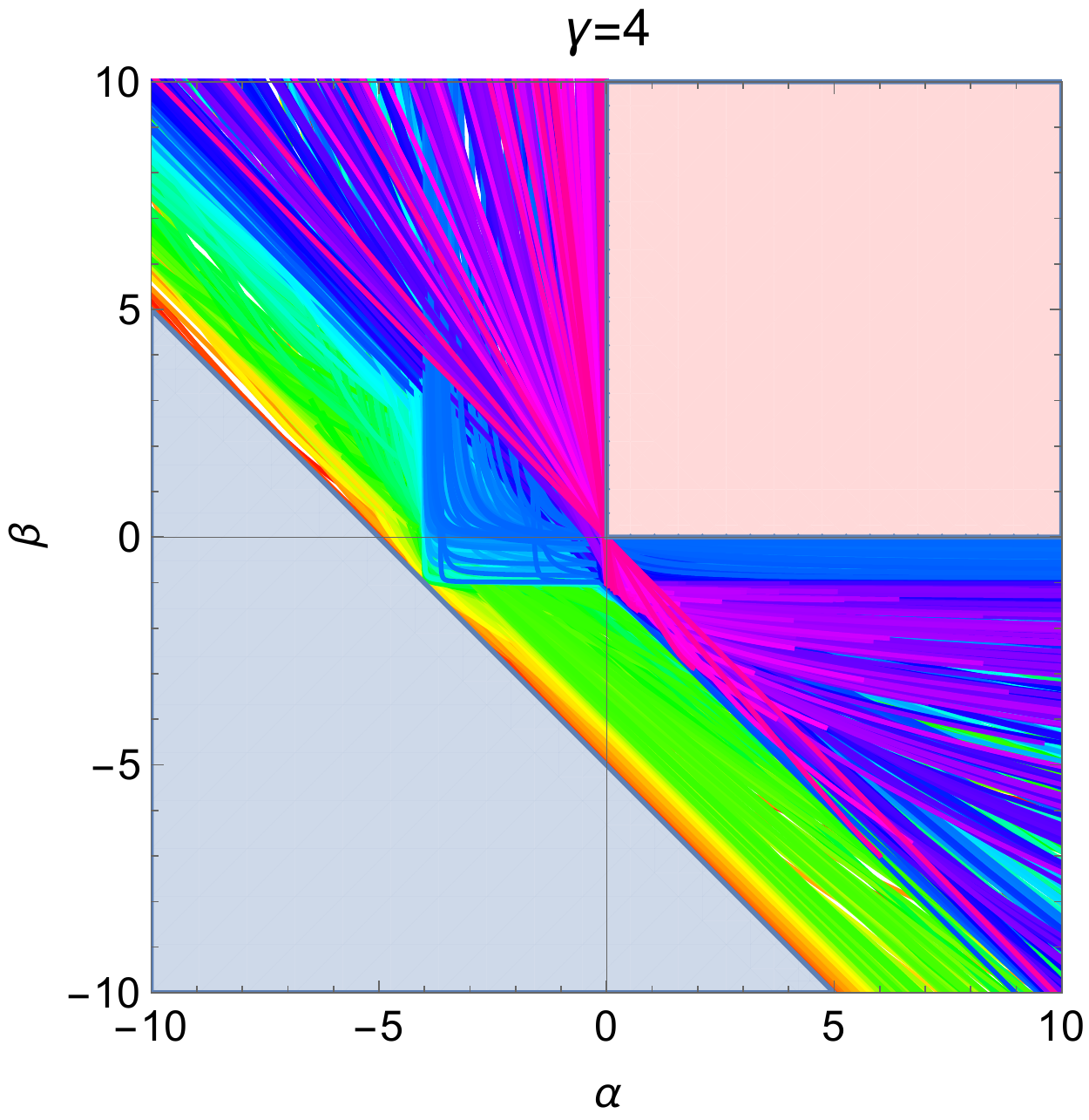}
    \caption{Curves $\Gamma(\gamma,q_1,q_2,q)$ given by Lemma \ref{lem.curve.gamma}, for $\gamma=4$ and $11^3$ randomly generated triplets of fractional orders such that $0<q_1<q_2<q<2$ (different colors are assigned to each curve, depending on the value of $q_1+q_2+q$). The red / blue regions represent the fractional-order independent asymptotic stability / instability regions $S(\gamma)$ and $U(\gamma)$, respectively.}
    \label{fig.curves.all}
\end{figure}

\begin{lemma}\label{lem.Q.multi}
    We have $$Q(\gamma)=\mathbb{R}^2\setminus \left( S(\gamma)\cup U(\gamma) \right).$$
\end{lemma}

\begin{proof} We will prove this result by double inclusion.

\noindent \emph{Step 1. Proof of the inclusion $Q(\gamma)\subseteq\mathbb{R}^2\setminus\left(S(\gamma)\cup U(\gamma)\right)$.}

Let $\Gamma(\gamma,q_1,q_2,q)\in Q(\gamma)$. From Lemma \ref{lem.curve.gamma} (ii) it follows that the curve $\Gamma(\gamma,q_1,q_2,q)$ does not intersect the first quadrant of the $(\alpha,\beta)$-plane. Therefore $Q(\gamma)\cap S(\gamma)=\varnothing$.

Let $(\alpha,\beta)\in \Gamma(\gamma,q_1,q_2,q)$ and $\omega>0$ such that $$\alpha=\gamma^{1-\frac{q_1}{q}} h(\omega,q_1,q_2,q)~\text{and}~\beta=\gamma^{1-\frac{q_2}{q}} h(\omega,q_2,q_1,q).$$

We have
\begin{align*}
    \alpha+\beta&=\gamma^{1-\frac{q_1}{q}} h(\omega,q_1,q_2,q)+\gamma^{1-\frac{q_2}{q}} h(\omega,q_2,q_1,q)\\
    &=\gamma^{1-\frac{q_1}{q}}\omega^{-\frac{q_1}{q}}(\omega\tilde{\rho_2}-\rho_2)+\gamma^{1-\frac{q_2}{q}}\omega^{-\frac{q_2}{q}}(\rho_1-\omega\tilde{\rho_1}) \\
    &=(\gamma\omega)^{1-\frac{q_1}{q}}\tilde{\rho_2}-(\gamma\omega)^{1-\frac{q_2}{q}}\tilde{\rho_1}+\gamma\left[ (\gamma\omega)^{-\frac{q_2}{q}}\rho_1-(\gamma\omega)^{-\frac{q_1}{q}}\rho_2 \right]
\end{align*}

Denoting  $t=\gamma\omega>0$ and  $$u(t)=t^{1-\frac{q_1}{q}}\tilde{\rho_2}-t^{1-\frac{q_2}{q}}\tilde{\rho_1}\quad  \text{and}\quad v(t)=t^{-\frac{q_2}{q}}\rho_1-t^{-\frac{q_1}{q}}\rho_2,$$
we can easily see that $\alpha+\beta=u(t)+\gamma v(t)$.

The minimum of the function $u(t)$ is reached at $t^*=\left( \dfrac{(q-q_2)\tilde{\rho_1}}{(q-q_1)\tilde{\rho_2}} \right)^{\frac{q}{q_2-q_1}}$ and therefore
$$u(t)\geq u_{min}=u(t^*)=-\left( \dfrac{q-q_2}{\sin(q-q_2)\frac{\pi}{2}} \right)^{\frac{q-q_2}{q_2-q_1}} \left( \dfrac{\sin(q-q_1)\frac{\pi}{2}}{q-q_1} \right)^{\frac{q-q_1}{q_2-q_1}} \dfrac{q_2-q_1}{\sin (q_2-q_1)\frac{\pi}{2}}.$$
In what follows, we show that $u_{min}\geq -1$. As the function  $w(x)=x\ln\left(\frac{x}{\sin x}\right)$ is convex on $[0,\pi]$ with $\lim\limits_{x\to 0}w(x)=0$, we deduce that $q$ is superadditive and so we have $$w(x)+w(y-x)<w(y),~0<x<y<\pi.$$
Taking $x=(q-q_2)\frac{\pi}{2}$ and $y=(q-q_1)\frac{\pi}{2}$, we further obtain obtain $$\dfrac{q-q_2}{q_2-q_1}\ln \dfrac{q-q_2}{\sin (q_2-q_1)\frac{\pi}{2}}-\dfrac{q-q_1}{q_2-q_1}\ln\dfrac{q-q_1}{\sin (q-q_1)\frac{\pi}{2}}+\ln\dfrac{q_2-q_1}{\sin(q_2-q_1)\frac{\pi}{2}}< 0,$$ which leads to $u_{min}> -1$, implying that $u(t)> -1$, for all $t>0$.

Similar arguments lead to the inequality $v(t)> -1$, for all $t>0$.

Therefore,
$$\alpha+\beta=u(t)+\gamma v(t)> -1-\gamma,$$  which means that $(\alpha,\beta)\notin U(\gamma)$.
Hence $Q(\gamma)\subseteq \mathbb{R}^2\setminus(S(\gamma)\cup U(\gamma))$.

\medskip
\noindent \emph{Step 2. Proof of the inclusion $Q(\gamma)\supseteq\mathbb{R}^2\setminus\left(S(\gamma)\cup U(\gamma)\right)$.}

Consider the function $G_{\gamma}:(0,\infty)\times D\to \mathbb{R}^2$ defined by $$G_{\gamma}(\omega,q_1,q_2,q)=\left( \gamma^{1-\frac{q_1}{q}}h(\omega,q_1,q_2,q),\gamma^{1-\frac{q_2}{q}}h(\omega,q_2,q_1,q) \right),$$ where $D=\{ (q_1,q_2,q)\in\mathbb{R}^3: 0<q_1<q_2<q\leq 2 \}$ and the function $h$ defined in Lemma \ref{lem.curve.gamma}.

We can easily notice that $Q(\gamma)$ is the image of the function $G_{\gamma}$, i.e. $Q(\gamma)=G_{\gamma}((0,\infty)\times D)$.

As $h(-\omega,q_1,q_2,q)=h(\omega,q_2,q_1,q)$, it results that $Q(\gamma)$ is symmetric with respect to the first bisector $\alpha=\beta$ of the $(\alpha,\beta)$-plane.
Consequently, to determine $Q(\gamma)$, it suffices to characterise its intersection with an arbitrary straight line, parallel to the first bisector of the $(\alpha,\beta)$-plane: $$\alpha-\beta=m,\quad m\in\mathbb{R}.$$

We recall from Lemma \ref{lem.curve.gamma} that for fixed $\gamma>0$ and $(q_1,q_2,q)\in D$, the curve $\Gamma(\gamma,q_1,q_2,q)$ is the graph of a decreasing, smooth, convex bijective function in the $(\alpha,\beta)$-plane. Therefore, its intersection with the line $\alpha-\beta=m$ is exactly one point, which means that for arbitrary fixed $(q_1,q_2,q)\in D$ and $m\in\mathbb{R}$, the equation $$\gamma^{1-\frac{q_1}{q}}h(\omega,q_1,q_2,q)-\gamma^{1-\frac{q_2}{q}}h(\omega,q_2,q_1,q)=m$$ has a unique solution $\omega_m^*(q_1,q_2,q)$.

By the implicit function theorem and the properties of the function $h$, it results that $\omega_m^*$ is a continuously differentiable function on $D$. It follows that the abscissa of the intersection point of the curve $\Gamma(\gamma,q_1,q_2,q)$ and the line $\alpha-\beta=m$ is $$\alpha_m(q_1,q_2,q)=\gamma^{1-\frac{q_1}{q}}h(\omega_m^\star(q_1,q_2,q),q_1,q_2,q),$$ which is a continuously differentiable function on the set $D$, meaning that $\alpha_m(D)$ is an interval. With the aim of finding this interval, we will determine the extreme values of $\alpha_m$ over the set $D$. 

Applying Lagrange's multipliers method, we will show that $\alpha_m$ does not have any critical points in the interior of the set $D$. Indeed, considering the system $$\lambda \nabla \beta=(\lambda-1)\nabla \alpha,$$
and eliminating $\lambda$, we obtain $$\begin{cases}
& \dfrac{\partial \beta}{\partial \omega}\cdot \dfrac{\partial \alpha}{\partial q_1}=\dfrac{\partial \beta}{\partial q_1}\cdot \dfrac{\partial \alpha}{\partial \omega}\\
& \dfrac{\partial \beta}{\partial \omega}\cdot \dfrac{\partial \alpha}{\partial q_2}=\dfrac{\partial \beta}{\partial q_2}\cdot \dfrac{\partial \alpha}{\partial \omega}\\
& \dfrac{\partial \beta}{\partial \omega}\cdot \dfrac{\partial \alpha}{\partial q}=\dfrac{\partial \beta}{\partial q}\cdot \dfrac{\partial \alpha}{\partial \omega}
\end{cases}$$
where the arguments were dropped for simplicity. Taking the first two equations from the previous system and eliminating $\ln (\gamma \omega)$, we obtain a quadratic equation in $\omega$ with a negative discriminant, which implies that it does not have any real roots.

Hence, the extreme values of the function $\alpha_m$ can only be reached on the boundary of the set $D$, or equivalently, the boundary $\partial Q(\gamma)$ is made up of points belonging to the curve $\Gamma(\gamma,q_1,q_2,q)$ with $(q_1,q_2,q)\in \partial D$.

Therefore, it remains to show that $\partial Q(\gamma)=\partial S(\gamma)\cup \partial U(\gamma)$.

On one hand, based on the definition of the function $h(\omega,q_1,q_2,q)$ given in Lemma \ref{lem.curve.gamma}, it is easy to see that if $(q_1,q_2,q)\rightarrow(0,q^*_2,2)$ with $q^*_2\in(0,2)$, or $(q_1,q_2,q)\rightarrow(q^*_1,2,2)$ with $q^*_1\in(0,2)$,  the curve $\Gamma(\gamma,q_1,q_2,q)$ approaches either the semi-axis $\beta=0$, $\alpha>0$, or the semi-axis  $\alpha=0$, $\beta>0$, respectively. Hence,  $\partial S(\gamma)\subset \partial Q(\gamma)$.

On the other hand, considering $\omega=\gamma^{-1}$ in the parametric equations of the curve $\Gamma(\gamma,q_1,q_2,q)$ given in Lemma \ref{lem.curve.gamma}, we deduce that the points $P(\gamma,q_1,q_2,q)=(\tilde{\rho_2}-\gamma\rho_2,\gamma\rho_1-\tilde{\rho_1})\in Q(\gamma)$. For arbitrary $p_1,p_2,p>0$ such that $p_1<p_2<p$, let us consider the sequence of points $$P_n=P\left(\gamma,\frac{p_1}{n},\frac{p_2}{n},\frac{p}{n}\right)\in Q(\gamma), \quad n\in \mathbb{N},~n\geq \left\lceil\frac{p}{2} \right\rceil.$$
By the L'Hospital's rule, we obtain
$$\lim_{n\rightarrow\infty} P_n=\left(\frac{p-p_2(1+\gamma)}{p_2-p_1},\frac{p_1(1+\gamma)-p}{p_2-p_1}\right).$$
Consequently, the locus of the limit points obtained above is in fact the straight line $\alpha+\beta+\gamma+1=0$. Therefore,  $\partial U(\gamma)\subset \partial Q(\gamma)$.

Hence, the proof is complete.
\end{proof}

We now state the main results of this section, emphasizing that fractional-order independent stability and instability results are particularly important when the exact fractional orders of the derivatives appearing in equation \eqref{eq.multi.order} are not known precisely. This is closely linked to the fact that often, mathematical modeling of real world phenomena leads to fractional-order differential equations with estimated or conjectured fractional orders.  

\begin{theorem}[Fractional-order independent stability and instability results]\label{thm.instab.multi}$ $
\begin{itemize}
\item[i.] The multi-term fractional differential equation \eqref{eq.multi.order} is unstable, regardless of the fractional orders $(q_1,q_2,q)\in D$, if and only if 
$$ 
\gamma<0\quad\text{or}\quad\alpha+\beta+\gamma+1 \leq 0. 
$$
\item[ii.]  The multi-term fractional differential equation \eqref{eq.multi.order} is asymptotically stable, regardless of the fractional orders $(q_1,q_2,q)\in D$, if and only if 
$$\alpha>0,~\beta>0 \text{ and } \gamma>0.$$
\end{itemize}
\end{theorem}

\begin{proof} The proof of 
sufficiency of statements (i) and (ii) is provided by Lemmas \ref{unstab.lem} and \ref{stab.lem}. 

For the proof of necessity of statement (i), if $\gamma>0$ and the equation \eqref{eq.multi.order} is unstable, regardless of the fractional orders $(q_1,q_2,q)\in D$,  assuming by contradiction that $(\alpha,\beta)\notin U(\gamma)$, Lemma \ref{lem.Q.multi} implies that there exist $(q_1^*,q_2^*,q^*)\in D$ for which $(\alpha,\beta)$ belongs to the connected component of $\mathbb{R}^2\setminus \Gamma(\gamma,q_1^*,q_2^*,q^*)$ which includes $S(\gamma)$, or equivalently, $(\alpha,\beta)$ is above  the curve $\Gamma(\gamma,q_1^*,q_2^*,q^*)$. Therefore, Theorem \ref{thm.q} implies that for the particular fractional orders $q_1^*,q_2^*,q^*$, the equation \eqref{eq.multi.order} is asymptotically stable, which contradicts our claim. 

The proof of necessity of statement (ii) is completed by \textit{reduction ad absurdum} in a similar way as above, and hence, it will be omitted. 
\end{proof}

\section{Examples}

As applications to the theoretical results discussed in the previous sections, we first approach the well known Basset and Bagley-Torvik equations, which are in fact particular cases of our general equation \eqref{eq.multi.order}.

\begin{ex}\label{ex.basset}
The Basset equation.

We consider the linear fractional-order differential equation 
\begin{equation}\label{basset}
    \dot{x}(t)+\alpha ^c\!D^{q_1}x(t)+\gamma x(t)=0, \quad 0<q_1<1, 
\end{equation}
which arises in the investigation of the generalised Basset force that occurs when a spherical object sinks in an incompressible viscous fluid \cite{mainardi1997fractional}. It is clear that this is a particular case of equation \eqref{eq.multi.order} with $\beta=0$ and $q=1$.

The characteristic equation associated to \eqref{basset} is $$s+\alpha s^{q_1}+\gamma=0.$$

If $\gamma<0$, from Theorem \ref{thm.instab.multi} (i.) we have that the trivial solution of \eqref{basset} is unstable, regardless of the fractional order $q_1$. On the other hand, if $\alpha>$ and $\gamma>0$, by applying Theorem \ref{thm.instab.multi} (ii.), it follows that the trivial solution of equation \eqref{basset} is asymptotically stable, for any fractional order $q_1$.

Let us now assume that $\alpha\leq 0$ and $\gamma>0$. In this case, the stability of the linear equation \eqref{basset} depends on the fractional order $q_1\in(0,1)$. Indeed, considering the parametric equations of the curve $\Gamma(\gamma,q_1,q_2,1)$ defined in Lemma \ref{lem.curve.gamma}, taking into account that in this particular case $\beta=0$, we deduce that $\omega=\tan \frac{q_1\pi}{2}$. Replacing $\omega$ in the parametric equation for $\alpha$, we obtain the critical value \[\alpha^*(\gamma,q_1)=-\gamma^{1-q_1}\left(\cot \frac{q_1\pi}{2}\right)^{q_1}\sec \frac{q_1\pi}{2},\] 
which gives the same formula as in \cite{cermak2015stability} (in particular, for $q_1=\frac{1}{2}$, we have $\alpha^*(\gamma,\frac{1}{2})=-\sqrt{2\gamma}$). 
Then, following from Theorem \ref{thm.q}, it results that the trivial solution of equation \eqref{basset} is $\mathcal{O}(t^{-q_1})$-asymptotically stable if and only if $\alpha>\alpha^*(\gamma,q_1)$. 
\end{ex}

\begin{ex}\label{ex.bagley}
The Bagley-Torvik equation.

The following fractional-order differential equation arises when modeling the motion of a rigid plate that immerses in a viscous liquid \cite{torvik1984appearance}:
\begin{equation}\label{bagley}
    x''(t)+\alpha ^c\!D^{q_1}x(t)+\gamma x(t)=0,\quad 0<q_1<2
\end{equation}
where the term $\alpha ^c\!D^{q_1}x(t)$ is a fractional damping term which models the damping force in vibrating systems in viscous fluids. This is in fact a particular case of equation \eqref{eq.multi.order} with $\beta=0$ and $q=2$.

The characteristic equation associated to equation \eqref{bagley} is $$s^2+\alpha s^{q_1}+\gamma=0.$$

By a similar reasoning as in Example \ref{ex.basset}, the critical value for the parameter $\alpha$ is found to be $\alpha^*(\gamma,q_1)=0$, for any fractional order $q_1$ and any $\gamma>0$. Hence, by Theorems \ref{thm.q} and \ref{thm.instab.multi}, the trivial solution of equation \eqref{bagley} is asymptotically stable (of order $\mathcal{O}(t^{-\{q_1\}})$) if and only if $\alpha>0$ and $\gamma>0$, which is in accordance with the results presented in \cite{cermak2014exact}. 
\end{ex}

Further examples are given below, discussing multi-term fractional order differential equations which have been recently investigated in the scientific literature. 

\begin{ex} The inextendible pendulum.

The equation of motion of the inextendible pendulum \cite{seredynska2005nonlinear} is 
\begin{equation}\label{pendulum}
    {\phi}''+\mu\tau^{q_1}\cdot ^c\!\!D^{q_1}\phi+\nu\tau^{q_2}\cdot^c\!\!D^{q_2}\phi+\dfrac{g}{L}\sin \phi=0,
\end{equation}
Linearizing equation \eqref{pendulum} at the equilibrium $\phi_0=0$, we obtain
\begin{equation}\label{eq.pendulum}
    {\phi}''+\mu\tau^{q_1}\cdot  ^c\!\!D^{q_1}\phi+\nu\tau^{q_2}\cdot  ^c\!\!D^{q_2}\phi+\dfrac{g}{L} \phi=0.
\end{equation}
This is a particular case of equation \eqref{eq.multi.order} with $q=2$, $\alpha=\mu\tau^{q_1}>0$, $\beta=\nu\tau^{q_2}>0$, $\gamma=\frac{g}{L}>0.$
From Theorem \ref{thm.instab.multi} (ii.), it follows that the trivial equilibrium of equation \eqref{eq.pendulum} is asymptotically stable, regardless of the fractional orders $q_1$ and $q_2$. Moreover, Theorem \ref{thm.q} provides that
\[\phi(t)=\mathcal{O}(t^{-q'})\quad\text{as }t\rightarrow\infty,\quad\text{where }q'=\min\{\{q_1\},\{q_2\}\}.\]
\end{ex}

\begin{ex} Fractional harmonic oscillator. 

We consider the fractional differential equation of a linearly damped oscillator:
\begin{equation}\label{oscillator}
    x''+\alpha ^c\!D^{q_1}x +\beta x' +\gamma x=0,\quad\text{with }q_1\in(0,1)
\end{equation}
where $\alpha$ and $\beta$ are the friction coefficient and viscous damping coefficient, respectively \cite{chen2013stationary,guo2016stochastic}, and $\gamma>0$. This is in fact a particular case of equation \eqref{eq.multi.order} with $q=2$ and $q_2=1$.

The smooth parametric curve defined in Lemma \ref{lem.curve.gamma} becomes

\begin{equation}\label{gamma.ex4}
\Gamma(\gamma,q_1,1,2):\quad \begin{cases} \alpha&=\gamma^{1-\frac{q_1}{2}}\omega^{-\frac{q_1}{2}}(\omega-1)\sec \frac{q_1\pi}{2}\\
\beta&=\gamma^{\frac{1}{2}}\omega^{-\frac{1}{2}}(\omega-1)\tan \frac{q_1\pi}{2}
\end{cases}
\end{equation}

Moreover, as a particular case for $\gamma=4$, several curves $\Gamma(4,q_1,1,2)$ are plotted in Fig. \ref{fig.curves.ex4}, for different values of the fractional order $q_1\in(0,1)$. As in the general case explored in the previous sections, it can be noticed that the plotted curves do not intersect neither the stability region $S(\gamma)$, nor the instability region $U(\gamma)$.

\begin{figure}[htbp]
    \centering
    \includegraphics[width=0.5\linewidth]{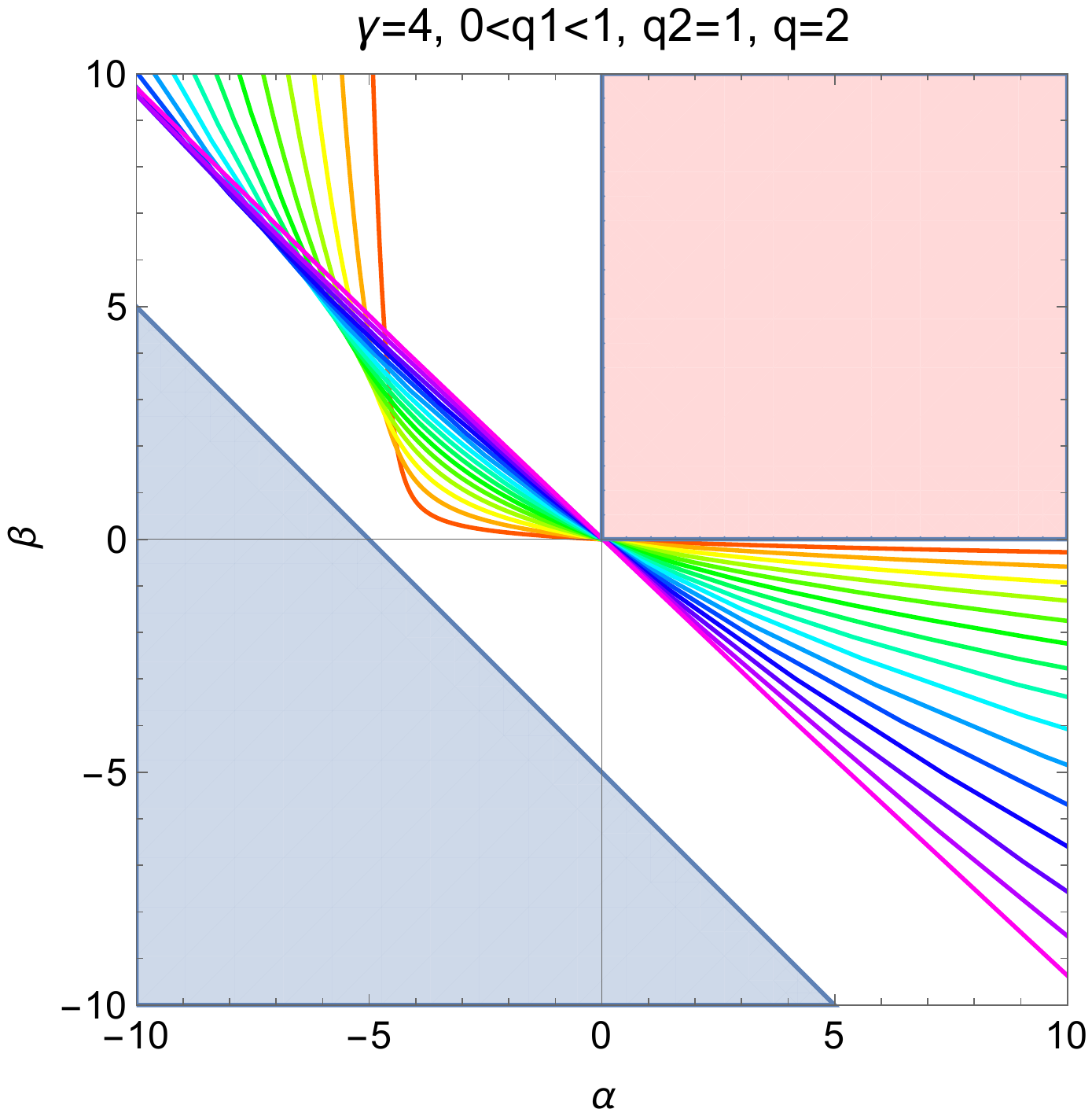}
    \caption{Curves $\Gamma(\gamma,q_1,1,2)$ given by \eqref{gamma.ex4}, for $\gamma=4$ and different values for the fractional order $q_1=k/16$, $k=\overline{1,15}$ (assigned colors range from red to purple with increasing values of $q_1$). The red / blue regions represent the fractional-order independent asymptotic stability / instability regions $S(\gamma)$ and $U(\gamma)$, respectively.}
    \label{fig.curves.ex4}
\end{figure}

Theorem \ref{thm.instab.multi} implies that if $\alpha>0$, $\beta>0$ and $\gamma>0$, the equation \eqref{oscillator} is asymptotically stable, regardless of the choice of the fractional order $q_1$. However, if for instance, a negative damping coefficient $\beta$ is taken into account, the stability properties of equation \eqref{oscillator} depend on the fractional order $q_1$ as well as the magnitude of the parameters. 

For example, if $\alpha=1.2$, $\beta=-1$ and $\gamma=4$, the critical value of the fractional order can be found by numerically solving the parametric equations \eqref{fig.curves.ex4} for $\omega$ and $q_1$, leading to $q_1^\star=0.81695$. It follows that in this case,  equation \eqref{oscillator} is asymptotically stable if and only if $q_1>q_1^\star$. The numerical solutions \cite{diethelm2004numerical} of equation \eqref{fig.curves.ex4} with respect to several values of $q_1$ are plotted in Figure \ref{fig.sim.ex4}, showing that when $q_1<q_1^\star$, the null solution of \eqref{oscillator} becomes unstable. It can also be noticed that as $q_1$ increases above the critical value $q_1^\star$, so does the rate of convergence to the trivial equilibrium. 

\begin{figure}[htbp]
    \centering
    \includegraphics[width=0.5\linewidth]{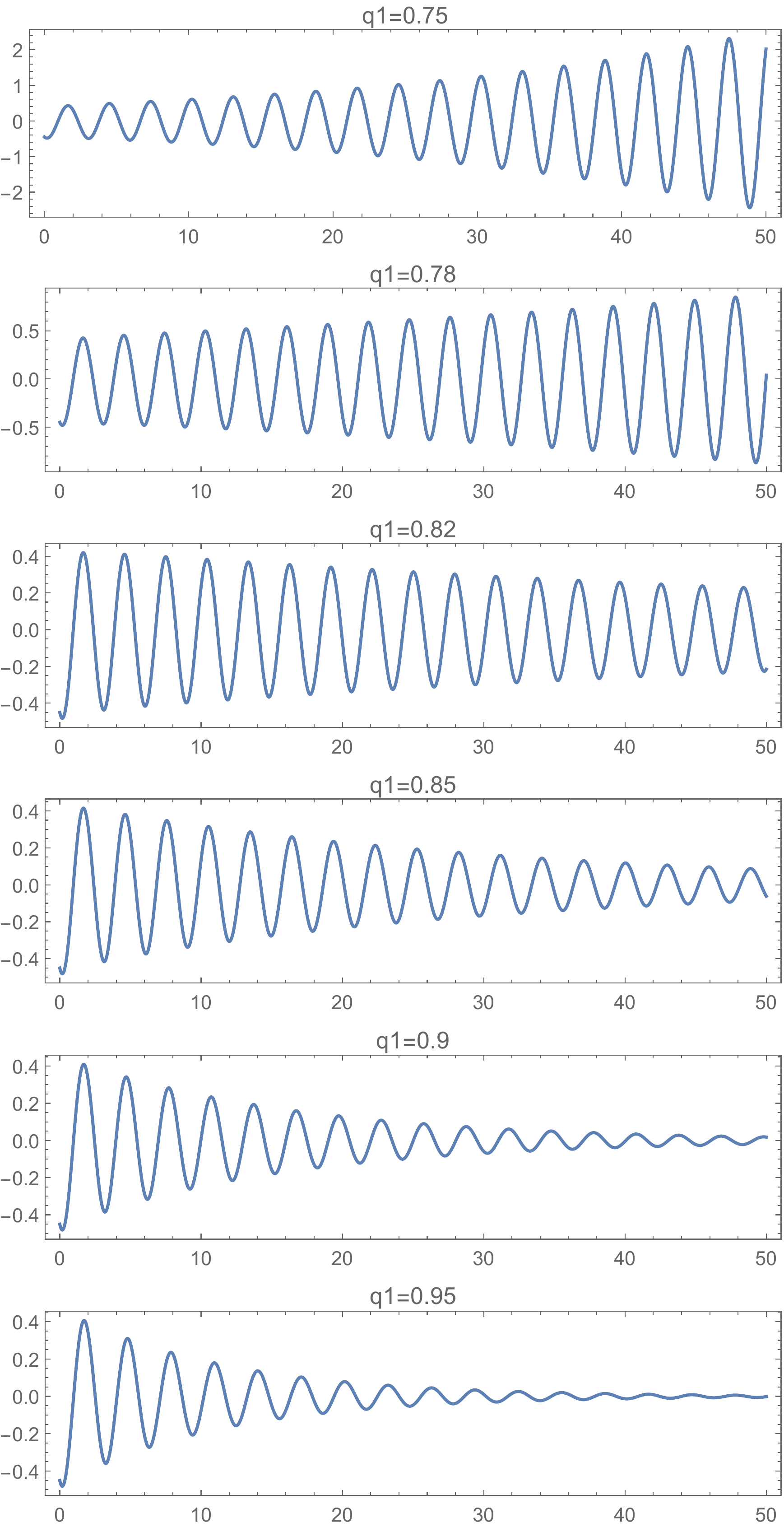}
    \caption{Numerical solutions of \eqref{oscillator} for $\alpha=1.2$, $\beta=-1$, $\gamma=4$ for several values of the fractional order $q_1$.}
    \label{fig.sim.ex4}
\end{figure}

\end{ex}

\section{Conclusions}

Multi-term linear fractional-order differential equations including three fractional-order derivatives have been analyzed, in extension to some recent results that have been obtained for linear incommensurate bidimensional autonomous fractional-order systems. Necessary and sufficient conditions have been obtained for the fractional-order-dependent and fractional-order-independent stability and instability of the considered equation. It has been shown that the obtained theoretical results can be succesfully applied for the particular cases of the Basset and Bagley-Torvik equations, for a multi-term fractional differential equation describing an inextensible pendulum with fractional damping terms as well as for a fractional harmonic oscillator. 

Possible generalizations to the case of general multi-term fractional-order differential equations will be investigated in a future work. 
\bibliographystyle{plain}

\bibliography{bibliografie}
\end{document}